\documentclass[a4paper,11pt]{amsart}
\usepackage{amsfonts}
 \usepackage{enumerate}
\usepackage{amsmath, amsthm}
\usepackage{amscd}
\usepackage{amssymb}
\usepackage[title]{appendix}
\usepackage{xy}
\xyoption{all}
\usepackage{latexsym,graphicx}
\usepackage[latin1]{inputenc}
\usepackage{xspace}
\usepackage{xcolor}
\usepackage{array}
\usepackage{newclude}
\usepackage{hyperref}
\hypersetup{pdfpagelabels,
                  plainpages=false,
                    colorlinks=true,       
                    linkcolor=red}
\renewcommand*{\HyperDestNameFilter}[1]{\jobname-#1} 
\numberwithin{equation}{section}
\usepackage[centering]{geometry}

\usepackage[nameinlink]{cleveref}

\newcommand{\sspace}{\vspace{0.25cm}}

 \theoremstyle{plain}
\newtheorem{theor}{Theorem}[section]

\newtheorem{prop}[theor]{Proposition}
\newtheorem{lem}[theor]{Lemma}

\theoremstyle{remark}
\newtheorem{rem}[theor]{Remark}
\newtheorem{rems}[theor]{Remarks}
\newtheorem{Example}[theor]{Example}

\theoremstyle{plain}
\newtheorem{defi}[theor]{Definition}

\numberwithin{equation}{section}

\newcommand{\CC}{{\mathbb C}}
\newcommand{\RR}{{\mathbb R}}

\newcommand{\QQ}{{\mathbb Q}}
\newcommand{\FF}{{\mathcal{ F}}}
\newcommand{\ZZ}{{\mathbb Z}}
\newcommand{\VV}{{\mathbb V}}

\newcommand{\G}{{\mathbf G}}
\newcommand{\HH}{{\mathbf H}}

\newcommand{\PP}{{\mathbf P}}

\newcommand{\Z}{{\mathbf Z}}
\newcommand{\C}{{\mathbf C}}
\newcommand{\NN}{{\mathbb N}}

\newcommand{\Ga}{\Gamma}

\newcommand{\HL}{\textnormal{HL}}
\newcommand{\inter}{\textnormal{Int}}

\newcommand{\ti}[1]{\mbox{$\tilde{#1} $}}

\newcommand{\ol}{\overline}

\newcommand{\lo}{\longrightarrow}

\newcommand{\End}{{\rm End}\,}

\newcommand{\ad}{{\rm ad}}

\newcommand{\GL}{{\rm \bf GL}}

\newcommand{\MT}{{\rm \bf MT}}
\newcommand{\g}{{\mathfrak{g}}}

\newcommand{\alg}{\textnormal{alg}}

\newcommand{\F}{\mathcal{F}}

\newcommand{\an}{\textnormal{an}}

\newcommand{\bH}{{\mathbf H}}

\def\FS{\mathfrak{S}}

\def\FH{\mathfrak{H}}

\def\Fg{\mathfrak{g}}

\def\Fn{\mathfrak{n}}

\newcommand{\cF}{\mathcal{F}}
\newcommand{\cA}{{\mathcal A}}

\newcommand{\cR}{{\mathcal R}}
\newcommand{\cS}{{\mathcal S}}

\newcommand{\cO}{{\mathcal O}}

\newcommand{\bP}{\mathbf P}
\newcommand{\bQ}{\mathbf Q}
\newcommand{\bL}{\mathbf L}
\newcommand{\bN}{\mathbf N}
\newcommand{\bM}{\mathbf M}
\newcommand{\bS}{\mathbf S}
\newcommand{\bz}{\mathbf z}

\newcommand{\oQ}{\overline{\QQ}}

\newcommand{\SL}{{\mathbf{SL}}}

\newcommand{\Hod}{\textnormal{Hod}}
\newcommand{\Reel}{\textnormal{Re}\,}
\newcommand{\Ima}{\textnormal{Im}\,}

\newcommand{\cl}{\textnormal{cl}}

\begin{document}
\title{Tame topology of arithmetic quotients and algebraicity of Hodge
loci}
\author{B. Bakker}
\author{B. Klingler}
\author{J. Tsimerman}

\thanks{B.K. was partially supported by an Einstein Foundation's
  professorship.\\\indent B.B. was partially supported by NSF grants DMS-1702149 and DMS-1848049. }

\begin{abstract}
In this paper we prove the following results:

$1)$ We show that any arithmetic quotient of a homogeneous space admits a natural real 
semi-algebraic structure for which its Hecke correspondences are
semi-algebraic.  A particularly important example is given by Hodge varieties, which parametrize pure polarized integral Hodge structures.

$2)$ We prove that the period map associated to any pure polarized
variation of integral Hodge structures $\VV$ on a smooth complex 
quasi-projective variety $S$ is definable with respect to an o-minimal structure on the relevant Hodge variety induced by the above semi-algebraic structure. 

$3)$ As a corollary of $2)$ and of Peterzil-Starchenko's
o-minimal Chow theorem we recover that the Hodge locus of
$(S, \VV)$ is a countable union of algebraic subvarieties of $S$, a
result originally due to Cattani-Deligne-Kaplan. Our approach
simplifies the proof of Cattani-Deligne-Kaplan, as it does not use 
the full power of the difficult multivariable
$SL_2$-orbit theorem of Cattani-Kaplan-Schmid.
\end{abstract}

\maketitle
%\tableofcontents

\section{Introduction.}

\subsection{Arithmetic quotients}
Arithmetic quotients are real analytic manifolds of the form
$S_{\Gamma,G, M}:= \Gamma \backslash G/M$, for $\G$ a connected
semi-simple linear algebraic $\QQ$-group, $G:= \G(\RR)^+$ the real Lie group
connected component of the identity of $\G(\RR)$, $M \subset G$ a
connected compact subgroup and $\Gamma \subset \G(\QQ)^+:= \G(\QQ) \cap G$
a neat arithmetic group. By a morphism of arithmetic quotients
we mean a real analytic map $(\phi, g): S_{\Gamma', G', M'} \lo S_{\Gamma,G, M}$
of the form $\Gamma' h'M' \mapsto \Gamma \phi(h')  g M$ for 
some morphism $\phi: \G' \lo \G$ of semi-simple linear algebraic
$\QQ$-groups and some element $g \in G$ such that $\phi(\Ga') \subset
\Ga$ and $\phi(M') \subset gMg^{-1}$. 
 
Such quotients are ubiquitous in various parts of
mathematics. For $M= \{1\}$ the arithmetic quotients $
S_{\Gamma, G, \{1\}}= \Gamma \backslash G$ are the main players in homogeneous dynamics, for example Ratner's
theory \cite{Rat0}, \cite{Rat1}. For $K\subset G$ a maximal compact subgroup the
arithmetic quotients $S_{\Gamma, G, K}$ are the arithmetic riemannian locally symmetric
spaces, for instance the arithmetic hyperbolic manifolds $\Gamma
\backslash SO(n,1)^+/SO(n)$. They are intensively studied by differential geometers and group
theorists. When $G$ is moreover of Hermitian type then $S_{\Gamma, G,K}$ is a so-called 
arithmetic variety (also called a connected Shimura variety if $\Gamma$ is a
congruence subgroup): this is a smooth complex quasi-projective variety,
naturally defined over $\oQ$ in the Shimura case. The simplest
examples of connected Shimura varieties are the modular
curves $\Gamma_0(N) \backslash SL(2, \RR)/ SO(2)$. Connected Shimura varieties play a paramount role in
arithmetic geometry and the Langlands program. Much more generally,
the connected Hodge varieties are arithmetic quotients which play a crucial role in
Hodge theory as targets of period maps.

\subsection{Moderate geometry of arithmetic
  quotients} \label{tameness}

For $S_{\Ga, G, K}$ a connected Shimura variety, the study of the topological 
tameness properties of the uniformization map $\pi: G/M \lo S_{\Ga, G,
K}$ recently provided a crucial tool for solving longstanding
algebraic and arithmetic questions  (see \cite{Pil}, \cite{PT},
\cite{KUY}, \cite{Tsi}, \cite{KUY18},
\cite{MPT}). Here tameness is understood in the sense 
proposed by Grothendieck \cite[\textsection 5 and 6]{Gro} and
developed by model theory under the name ``o-minimal structure'' (see
below). The first goal of this paper is to develop a similar 
study for a general arithmetic quotient $S_{\Ga, G, M}$. 

Among real analytic manifolds the ones with the tamest geometry
are certainly the complex algebraic ones. However most
arithmetic quotients have no complex algebraic  
structures, as they do not even admit a complex analytic one (for 
instance for obvious dimensional reasons). What about a real algebraic
structure? In \cite{Rag} Raghunathan proved that any riemannian locally symmetric
space is compactifiable, i.e. diffeomorphic to the interior of a compact smooth manifold
with boundary; Akbulut and King \cite{AK} proved that any
compactifiable manifold is diffeomorphic to a non-singular real
algebraic set (generalizing a result of Tognoli \cite{Tog} in the
compact case, conjectured by Nash \cite{Nash}). Hence any riemannian
locally symmetric space is diffeomorphic to a non-singular
semi-algebraic set. On the other hand such  
abstract real algebraic models are useless if they don't satisfy some basic
functorial properties. A crucial feature of the geometry of arithmetic quotients is the
existence of infinitely many real-analytic finite self-correspondences: any
element $g \in \G(\QQ)$ commensurates $\Gamma$ (meaning that the
intersection $g\Gamma g^{-1} \cap \Gamma$ is of finite index in both
$\Gamma$ and $g\Gamma g^{-1}$) hence defines a Hecke correspondence 
\begin{equation} \label{Hecke}
\xymatrix{
 c_g = (c_1, c_2) :& S_{\Ga, G, M} &S_{g^{-1}\Gamma g \cap \Gamma, G, M} \ar[l]^>>>>>{c_1} 
\ar[r]^{g \cdot} \ar@/_1pc/[rr]_{c_2} & S_{\Gamma \cap g \Gamma g^{-1}, G, M}
\ar[r]&S_{\Ga, G, M}}\;\;.
\end{equation}
Here the left and right morphisms of arithmetic quotients are the natural finite \'etale
projections; the map in the middle is left-multiplication by $g$,
i.e. the morphism of arithmetic quotients $(\inter(g), g)$, where
$\inter(g):\G \lo \G$ denotes the conjugation by $g$. We
would like these Hecke correspondences to
be real algebraic. Such functorial real algebraic models do exist in certain cases: see
\cite{Jaf75}, \cite{Jaf78}, \cite{Le79}; but we don't know of any general procedure for producing
such a nice real algebraic structure on all arithmetic quotients. Hence our
need to work with a more general notion of tame geometry.

Recall that a structure $\cS$ on $\RR$ expanding the real field is a collection $(S_n)_{n \in \NN^*}$,
where $S_n$ is a set of subsets of $\RR^n$ (called the {\em $\cS$-definable
sets}), such that: all algebraic subsets of $\RR^n$ are in $S_n$; $S_n$ is  a boolean subalgebra of the power set of $\RR^n$;
if $A\in S_n$ and $B \in S_m$ then $A \times B \in
S_{n+m}$; if $p: \RR^{n+1} \lo \RR^n$ is a linear projection and $A
\in S_{n+1}$ then $p(A) \in S_n$. A function $f: A \lo B$ between
$\cS$-definable sets is
said to be $\cS$-definable if its graph 
is $\cS$-definable. Such a structure $\cS$ is said in addition to be {\em o-minimal} if the definable subsets of
$\RR$ are precisely the finite unions of points and intervals
(i.e. the semi-algebraic subsets of $\RR)$. This o-minimal axiom guarantees the possibility of doing
geometry using definable sets as basic blocks: it excludes infinite
countable sets, like $\ZZ \subset \RR$, as well as Cantor sets or
space-filling curves, to be
definable. Intuitively, subsets of
$\RR^n$ definable in an o-minimal structure are the ones having at
the same time a reasonable local topology and a tame topology at
infinity. Given an o-minimal structure $\cS$, there is an obvious
notion of $\cS$-definable manifold: this is a manifold $S$ admitting a {\em
  finite} atlas of charts $\varphi_i: U_i \lo \RR^n$, $i \in I$, such
that the intersections $\varphi_i(U_i \cap U_j)$, $i, j \in I$, are
$\cS$-definable subset of $\RR^n$ and the change of coordinates
$\varphi_i \circ \varphi_j^{-1}: \varphi_j(U_i \cap U_j) \lo
\varphi_i(U_i \cap U_j)$ are $\cS$-definable maps. 

The simplest o-minimal structure is $\RR_{\alg}$, the definable sets being the
semi-algebraic subsets. There exist more general o-minimal structures. A result of Van den
Dries based on Gabrielov's results \cite{Gabrielov} shows that the
structure $$\RR_\an:=\langle \RR,\, +, \,\times,\, <, \{f\} \, 
\textnormal{for} \, f \,\textnormal{restricted analytic function}
\rangle$$ generated from $\RR_\alg$ by adding the restricted analytic
functions is o-minimal. Here a real function on $\RR^n$ is restricted analytic
if it is zero outside $[0,1]^n$ and coincides on $[0,1]^n$ with a real analytic function
$g$ defined on a neighbourhood of $[0,1]^n$. The $\RR_\an$-definable
sets of $\RR^n$ are the globally subanalytic subsets of $\RR^n$
(i.e. the ones which are subanalytic in the compactification $\bP^n
\RR$ of $\RR^n$). A deep result of Wilkie
\cite{Wil} states that the structure $\RR_{\exp}:= \langle \RR, +,
\times, <, \exp: \RR \lo \RR \rangle$ generated from $\RR_{\alg}$ by making the real
exponential function definable is also o-minimal. Finally the
structure $\RR_{\an, \exp}:= \langle \RR,\, +, \,\times,\, <,
\,\exp, \{f\} \, \textnormal{for} \,  f \,\textnormal{restricted
  analytic function} \rangle$ 
generated by $\RR_{\an}$ and $\RR_{\exp}$ is still o-minimal \cite{VdM}.

\medskip
The first result of this paper is the following:
\begin{theor}
\label{definability}
Let $\G$ be a connected linear semi-simple algebraic $\QQ$-group,
$\Gamma \subset \G(\QQ)^+$  a torsion-free arithmetic lattice 
 of $G:= \G(\RR)^+$, and $M \subset G$ a connected compact subgroup.
\begin{itemize}
\item[(1)] \label{unif}
The arithmetic quotient $S_{\Gamma, G, M} := \Gamma \backslash G/M$ admits a natural structure of
$\RR_{\alg}$-definable manifold, characterized by the following
property. Let $G/M$ be endowed with its natural semi-algebraic
  structure (see \Cref{G/M semi}) and $\FS \subset G/M$ be a
  semi-algebraic Siegel set (see \Cref{Siegel sets} for the definition
  of Siegel sets). Then 
$$\pi_{|\FS}: \FS  \lo  S_{\Gamma,G, M} $$ is $\RR_{\alg}$-definable.

In particular, there exists a semi-algebraic fundamental set $\FF \subset G/M$ for the
  action of $\Gamma$ on $G/M$ such that 
$$\pi_{|\FF}: \FF \lo S_{\Gamma,G, M} $$ is $\RR_{\alg}$-definable. 

The structure of $\RR_\an$-definable manifold on $S_{\Gamma, G, M}$
extending its $\RR_{\alg}$-structure is the one induced by the
real-analytic structure with corners of its Borel-Serre compactification
$\overline{S_{\Ga, G, M}}^{BS}$. 
\item[(2)] Any morphism $f: S_{\Ga', G', M'} \lo S_{\Ga, G, M}$ of
  arithmetic quotients is $\RR_\alg$-definable. In particular the Hecke
  correspondences $c_g$, $g \in \G(\QQ)^+$, on $S_{\Ga, G, M}$ are $\RR_{\alg}$-definable. 
\end{itemize}
\end{theor}

\Cref{definability}(1) can be thought as
a strengthening and a generalization of the main result of
\cite{PetStar} (for $S_{\Gamma, G, K} = \cA_g$ the moduli space of principally polarized Abelian
varieties of dimension $g$) and of \cite[Theor.1.9]{KUY} (for a
general arithmetic variety), which proved that for any
arithmetic variety $S_{\Gamma, G, K}$, endowed with the
$\RR_{\an}$-definable manifold structure deduced from its complex algebraic Baily-Borel
compactification, there exists a semi-algebraic fundamental set $\FF \subset G/K$ for the
  action of $\Gamma$ on $G/K$ such that the map $\pi_{|\FF}: \FF \lo S_{\Gamma,G, K}$ is
$\RR_{\an, \exp}$-definable. While these results claim only the
definability of $\pi_{|\FF}$ in $\RR_{\an, \exp}$ our
\Cref{definability}(1) claim it in $\RR_\alg$. This discrepancy comes
from the fact that for $S_{\Gamma, G, K}$ an arithmetic variety, the
$\RR_\an$-definable structure on $S_{\Gamma, G, K}$ extending the
natural $\RR_\alg$-definable structure of \Cref{definability}(1) on
$S_{\Gamma, G, K}$ is the one coming from the Borel-Serre
compactification $\ol{S_{\Gamma, G, K}}^{BS}$ of $S_{\Gamma, G, K}$:
it does not coincide with the one coming from the Baily-Borel
compactification $\ol{S_{\Gamma, G, K}}^{BB}$ but the natural map
$\overline{S_{\Gamma, G, K}}^{BS} \lo \overline{S_{\Gamma, G,
    K}}^{BB}$ is in fact $\RR_{\an, \exp}$-definable. As we won't need this
result in this paper we just provide the simplest illustration:

\begin{Example} \label{Ex}
Let $\FH$ be the Poincar\'e upper half-plane and $Y_0(1)$ the modular
curve $\SL(2, \ZZ) \backslash \FH$. A semi-algebraic fundamental set for the
action of $\SL(2, \ZZ)$ on $\FH$ is given by $$\cF:=\{ x+iy \in \FH\; | \; x^2 + y^2 \geq 1, -1/2 \leq x\leq 1/2\}\;\;.$$
The Borel-Serre compactification $\ol{Y_0(1)}^{BS}$ is obtained by adding a circle at
infinity to $Y_0(1)$, corresponding to the compactification $\ol{\cF}$
of $\cF$ obtained by glueing the segment $\{ y= \infty, -1/2 < x<
1/2\}$ to $\cF$. The Baily-Borel compactification $X_0(1):= \ol{Y_0(1)}^{BB}$
is the one-point compactification of $Y_0(1)$ and is naturally
identified with the complex projective line $\bP^1\CC$. The natural map
$\ol{Y_0(1)}^{BS} \lo \ol{Y_0(1)}^{BB}$ contracting the circle at
infinity to a point sends a point $(x, t= 1/y) \in [-1/2, 1/2] \times [0, 1)$
close to the circle at infinity $t=0$ to the point $[1, z= \exp(2 \pi i x) \exp (-2
\pi/t)] \in \bP^1\CC$. This map is not globally subanalytic but it is
$\RR_{\an, \exp}$-definable.
\end{Example}

It is worth noticing that the proof of the more general
\Cref{definability} is easier than the one in \cite{PetStar} (which uses explicit theta
functions) or the one in \cite{KUY} (which uses the delicate toroidal
compactifications of \cite{AMRT}): it relies
exclusively on classical properties of Siegel sets (see \Cref{Siegel sets} for the precise
definition of Siegel sets), while the proofs of \cite{PetStar} and
\cite{KUY}, which apply only to arithmetic varieties,  
moreover insisted on using only complex analytic maps, thus obscuring to some
extent the o-minimality issues.

%%%%%%%%%%%%%%%%%%%%%%
\subsection{Moderate geometry of period maps} Arithmetic quotients
of interest to the algebraic geometers arise in Hodge theory as {\it
  connected Hodge varieties}, which are complex analytic quotients of
period domains (or more generally Mumford-Tate domains). Let $S$ be a
smooth complex quasi-projective variety and let $\VV
\lo S$ be a polarized variation of $\ZZ$-Hodge structures (PVHS) of
weight $k$ on $S$. A typical example of such a 
PVHS is $\VV = R^kf_*\ZZ$ for $f: \mathcal{X} \lo 
S$ a smooth proper morphism; in which case we say that $\VV$ is geometric.
We refer to \cite{K17} and the references therein for the
relevant background in Hodge theory, which we use thereafter. 
Let $\MT(\VV)$ be the generic Mumford-Tate group
associated to $\VV$ (this is a connected reductive $\QQ$-group) and $\G$
its associated adjoint semi-simple $\QQ$-group. The group $G:= \G(\RR)^+$ acts by
holomorphic transformations and transitively on the Mumford-Tate
domain $D=G/M$ associated to $\MT(\VV)$, with compact isotropy denoted
by $M$. If $\Ga$ is a torsion free arithmetic lattice of $G$ the arithmetic
quotient $S_{\Ga, G, M}$ is a complex analytic manifold called a
connected Hodge variety (which carries an algebraic structure in only very few
cases). Replacing if necessary $S$ by a finite \'etale cover, the PVHS 
$\VV$ on $S$ is completely described by its holomorphic period map
$\Phi_S: S \lo \Hod^0(S, \VV):= S_{\Ga, G, M}$ for a suitable
torsion-free arithmetic subgroup $\Gamma \subset G$.

We prove that the period map $\Phi_S$ has a moderate
geometry. Let us endow $S$ with the $\RR_{\an, \exp}$-definable manifold
structure extending the $\RR_{\alg}$-definable manifold structure on $S$ 
coming from its complex algebraic structure, and the connected Hodge
manifold $\Hod^0(S, \VV)= S_{\Ga, G, M}$ with the $\RR_{\an, \exp}$-definable manifold
structure extending the $\RR_{\alg}$-definable
manifold structure defined in \Cref{definability}.

\begin{theor} \label{definability period map}
Let $\VV $ be a polarized variation of pure Hodge
structures of weight $k$ over a smooth complex quasi-projective variety $S$.
Let $\Phi_S: S \lo \Hod^0(S, \VV)= S_{\Ga, G, M}$ be the 
holomorphic period map associated to $\VV$. 
Then $\Phi_S$ is $\RR_{\an, \exp}$-definable. 
\end{theor}

\begin{rems}
\begin{itemize}
\item[(1)] Notice that \Cref{definability period map} is easy in the
  rare case when the connected
  Hodge variety $S_{\Ga, G, M}$ is compact. In that case, consider $\ol{S}$ a
  smooth projective compactification of $S$ with normal crossing divisor at
infinity. It follows from Borel's monodromy theorem \cite[Lemma
(4.5)]{Schmid} and the fact that the cocompact lattice $\Ga$ does not
contain any unipotent element \cite[Cor. 11.13]{Rag72} that the
monodromy at infinity of $\VV$ is finite. Thus, replacing if
necessary $S$ by a finite \'etale cover, the PVHS $\VV$ extends to
$\ol{S}$. Equivalently the period map
$\Phi_S: S \lo \Hod^0(S, \VV):= S_{\Ga, G, M}$ extends to a
period map $\Phi_{\ol{S}}: \ol{S} \lo S_{\Ga, G, M}$. In particular
the period map 
$\Phi$ is definable in $\RR_\an$ in that case.

\item[(2)] When the connected Hodge variety $S_{\Ga, G, M}$ is an
  arithmetic variety, \Cref{definability period map} implies (see
  \Cref{algebraic}) that $\Phi_S: S \lo S_{\Ga, G, M}$ is an algebraic
  map, thus recovering a classical
  result due to Borel \cite[Theor. 3.10]{Bor72}. Hence \Cref{definability period map}
  can be thought as an extension of Borel's result to the general
  case where the connected Hodge variety $S_{\Ga, G, M}$ has no
  algebraic structure. On the other hand, notice that Borel
  \cite[Theor.A]{Bor72} proves in the arithmetic variety case the
  stronger result that $\Phi_S$ extends to a holomorphic map $\Phi_{\ol{S}}: \ol{S} \lo \ol{S_{\Ga, G,
      M}}^{BB}$, which does not directly follow from \Cref{definability period map}. 

%\item[(3)] A long standing conjecture of Griffiths, whose proof has recently been
%announced in \cite{GGLR}, states that $\Phi(S)$ admits a natural
%completion $\ol{\Phi(S)}$ as a projective algebraic variety (see
%\cite{Sommese} for earlier results in this direction) and that the map
%$\Phi_S: S \lo \Phi(S)$ extends to an algebraic map $\Phi_{\ol{S}}:
%\ol{S} \lo \ol{\Phi(S)}$. This would imply that $\Phi(S)$ has a tame
%topology, but says nothing about the tameness of $\Phi_S: S \lo
%S_{\Ga, G, M}$, as the relation
%between the projective compactification $\ol{\Phi(S)}$ of \cite{GGLR} and the Hodge
%variety $S_{\Ga, G,M}$ is far from clear. Hence \Cref{definability
%  period map} seems to go in a direction different from the one
%followed in \cite{GGLR}.
\end{itemize}
\end{rems}

The main ingredient in the proof of \Cref{definability period map} is
the following finiteness result on the geometry of Siegel sets:

\begin{theor} \label{finitely many Siegel} 
Let $\Phi: (\Delta^*)^n \lo S_{\Ga, G, M}$ be a local period
map with unipotent monodromy on a product of punctured disks. Let $\tilde{\Phi}: \FH^n 
\lo G/M$ be its lifting to the universal cover $\FH^n$ of $(\Delta^*)^n$. 
Given constants $R>0$ and $\eta >0$ let us define
\[\FH_{R,\eta}^n:=\{\bz\in\FH^n\;\mid \; |\Reel \bz| \leq C\mbox{ and }\Ima \bz \geq \eta\}\] 
where $|\Reel \bz| := \sup_{1\leq j \leq n} |\Reel z_i|$ and $\Ima \bz
:= \inf_{1 \leq j \leq n} \Ima z_i$.

There exists finitely many Siegel sets $\FS_i \subset G/M$, $i \in I$, such that
$$\tilde{\Phi}(\FH_{R,\eta}^n) \subset \bigcup_{i \in I}\FS_i\;\;.$$

\end{theor}

In the one-variable case ($n=1$) \Cref{finitely many Siegel} is due to Schmid (see
\cite[Cor. 5.29]{Schmid}), with $|I|=1$. In the multivariable case, Green, Griffiths, Laza and
Robles \cite[Claims A.5.8 and A.5.9]{GGLR} show that the result with
$|I|=1$ does not hold. 

The main point of the proof of \Cref{finitely many Siegel} is to show
there is a flat frame with respect to which the Hodge form remains
Minkowski reduced, up to covering $\FH^n_{R,\eta}$ by finitely many
sets.  We note here that the proof does not use the higher dimensional
$SL_2^n$-orbit theorem of \cite{CKS}.  Rather, we deduce the
higher-dimensional statement by restricting to curves and using the
full power of Schmid's one-dimensional result, together with the work
of \cite{CKS} and \cite{Ka} on the asymptotics of Hodge norms.

%%%%%%%%%%%%%%%%%%%%%%%%%%%%%%%%%%%%%%%%%%%%%%%%%%%%%%%%%%%%%%%%%%%%%%%%%%%%
\subsection{Algebraicity of Hodge loci}
Recall that the Hodge locus $\HL(S, \VV)
\subset S$ associated to the PVHS $\VV$ is the set of 
points $s$ in $S$ for which exceptional Hodge tensors for $\VV_{s}$
occur. The locus $\HL(S, \VV)$ is easily seen to be a countable union
of irreducible complex analytic subvarieties of $S$, called special
subvarieties of $S$ associated to $\VV$. If $\VV = R^kf_*\QQ$ for $f: \mathcal{X}
\lo S$ a smooth proper morphism, it follows from the Hodge
conjecture that the exceptional Hodge tensors in $\VV_{s}$ come from exceptional
algebraic cycles in some product $\mathcal{X}_s^N$. A Baire category
type argument then implies that every special subvariety of $S$ ought
to be algebraic. As an immediate
corollary of \Cref{definability}, \Cref{definability 
  period map}, and Peterzil-Starchenko's o-minimal Chow \Cref{PS} we
obtain an alternative proof of the following result originally proven by 
Cattani, Deligne and Kaplan \cite{CDK95}: 

\begin{theor}  \label{algebraicity}
The special subvarieties of $S$ associated to $\VV$ are algebraic, i.e.
the Hodge locus $\HL(S, \VV)$ is a countable union of closed irreducible algebraic
subvarieties of $S$.
\end{theor}

The proof of \Cref{algebraicity} in
\cite{CDK95} works as follows. Let $\ol{S}$ be a smooth
compactification of $S$ with a simple normal crossing divisor $D$ at
infinity. Locally in the analytic topology
$S$ identifies with $(\Delta^*)^r \times
\Delta^l$ inside $\ol{S} = \Delta^{r+l}$ (where $\Delta$ denotes the
unit disk). The ${SL_2}^n$-orbit theorems of \cite{Schmid} and 
\cite{CKS} describes extremely precisely the asymptotic of the period
map $\Phi_S$ on $(\Delta^*)^r \times
\Delta^l$. Using this description, Cattani, Deligne and Kaplan manage
to write sufficiently explicitly the equation of the locus $S(v)
\subset (\Delta^*)^r \times \Delta^l$ of the points at which some
determination of a given multivalued flat section $v$ of $\VV$ is a Hodge
class to prove that its closure $\ol{S(v)}$ in $\Delta^{r+l}$ is
analytic in this polydisk. Our proof via \Cref{definability period
  map} bypasses these delicate local computations, hence seems a worthwhile
simplification.

In view of \Cref{definability period map}, its corollary
\Cref{algebraicity}, and the recent proof \cite{BaT} (also using o-minimal techniques) of
the Ax-Schanuel conjecture for pure Hodge
varieties stated in \cite[Conj. 7.5]{K17}, we hope to convey the idea
that o-minimal geometry is an important tool in variational Hodge
theory. We refer to \cite[section 1.5]{K17} for 
possible applications of these results to the structure of $\HL(S,
\VV)$.

\Cref{algebraicity} has been extended to the case of (graded
polarizable, admissible) variation of mixed Hodge structures in \cite{BP1}, \cite{BP2},
\cite{BP3}, \cite{BPS}, \cite{KNU11} using \cite{CDK95} and the ${SL_2}^n$-orbit
theorem of \cite{KNU08} which  extends \cite{Schmid} and
\cite{CKS} to the mixed case. Our o-minimal proof of \Cref{definability period
  map} should certainly extend to this case, thus giving a simpler proof of the
algebraicity of Hodge loci in full generality. We will come back to
this problem in a sequel to this paper.

\subsection{Acknowledgments} B.K would like to thank Patrick Brosnan,
who asked him some time ago about a proof of \Cref{algebraicity} using o-minimal techniques; Mark Goresky,
Lizhen Ji and Arvind Nair, who made him notice that the map
from the Borel-Serre compactification to the Baily-Borel one is not
subanalytic, and that Borel-Serre compactifications are not
functorial; Colleen Robles, for the fruitful exchanges about
\Cref{finitely many Siegel}; Wilfried Schmid, who confirmed to him that \Cref{finitely many
  Siegel} should hold; and Kobi Peterzil and Sergei Starchenko,
whose comments led to the upgrading of \Cref{definability} from
$\RR_\an$ to $\RR_\alg$.  B.B. and J.T. would like to thank Wilfried Schmid and Yohan Brunebarbe for useful conversations.

%%%%%%%%%%%%%%%%%%%%%%%%%%%%%%%%%%%%
\section{Preliminaries}

\subsection{Semi-algebraic structure on $G/M$}
The existence of a natural semi-algebraic structure on the model $G/M$
of an arithmetic quotient, stated in the following lemma, is folkloric
but we did not find a precise reference. For the convenience of the reader we provide two proofs: a
``classical''  algebraic one, and an o-minimal one
announcing the proof of \Cref{definability}(1).

\begin{lem} \label{G/M semi}
Let $\G$ be a connected semi-simple linear algebraic $\QQ$-group, $G:= \G(\RR)^+$ the real Lie group
connected component of the identity of $\G(\RR)$, and $M \subset G$ a
connected compact subgroup. Then $G/M$ admits a natural structure of a
semi-algebraic set, and the projection map $G \lo G/M$ is
semi-algebraic. The action by left-multiplication of $G$ on $G/M$ is semi-algebraic.
\end{lem}

\begin{rem} \label{rem1}
In general $G/M$ does not admit a structure of real algebraic variety.
This is already true for $G$: for instance the group $SO(p,q)$ is a real
algebraic variety but its connected component $G:= SO(p,q)^+$
is only semi-algebraic for $p\geq q>0$. On the other hand any compact real Lie group
$M$ admits a natural structure $\bM_\RR$ of real algebraic group, see \cite[Th. 5, p.133]{OV}.
\end{rem}

%\begin{proof}
%Consider the action of $\G_\RR$ on itself by left multiplication. Then
% the stabilizer of the closed subvariety $\bH_\RR$ is $\bH_\RR$
% itself. Thus, $\bH_\RR$ is also the stabilizer of the ideal
%$I(\bH_\RR) \subset \RR[\G_\RR]$. Let us choose a finite-dimensional
% space $I \subset I(\bH_\RR)$ which generates that ideal. As
%$I(\bH_\RR)$ is a rational $\bH_\RR$-module we can also assume that $I$ is
%$\bH_\RR$-stable. Hence it is contained in a finite dimensional
%$\G_\RR$-submodule $J \subset \RR[\G_\RR]$. As $\bH_\RR$ is the
% stabilizer of $I$, it is also the stabilizer of the line $l:=\bigwedge^n
% I$ of the $\G_\RR$-module $W:= \bigwedge^nJ$, where $n := \dim(I)$.
% \end{proof}
\begin{proof}
  
Let us start with the algebraic proof, inspired by \cite{Schwarz}. By
a classical result of Chevalley \cite{Che}, there exists 
a finite dimensional $\G_\RR$-module $W$ and a line $l \subset W$ such 
that the stabilizer in $\G_\RR$ of $l$ is precisely $\bM_\RR$ (the real algebraic
subgroup of $\G_\RR$ such that $\bM_\RR(\RR)=M$).
As the group $M$ is compact connected, it not only stabilizes the line $l$ but fixes any
generator $v$ of $l$. 
By another classical result, this time due to Hilbert (see \cite[Ch. VIII, \textsection
14]{Weyl}), the (graded) algebra $\RR[W]^{\bM_\RR}$ of
$\bM_\RR$-invariant polynomials on $W$ is finitely generated, say by
homogeneous elements $p_1, \cdots, p_d$. Consider the real algebraic map
$p : \G(\RR) \rightarrow W \rightarrow \RR^d$
obtained by composing the orbit map of the vector $v \in W$ with $
(p_1, \cdots, p_d): W \lo \RR^d$. It identifies $\G(\RR)/M$ with the image
$p(\G(\RR))$, hence $G/M$ with a connected component of
$p(\G(\RR))$. As $p$ is real algebraic the subset $p(\G(\RR))$, hence
its connected component $G/M$, is semi-algebraic.
As $p$ is real-algebraic, the projection $G \lo G/M$ is
semi-algebraic.

Let us turn to the o-minimal proof. This is \cite[(2.18) p.167]{VDD},
which we summarize. The multiplication $G\times G\to G$ is the
restriction of an algebraic map hence semi-algebraic. The
equivalence relation  
\[E=\{(g,gm)\mid g\in G\mbox{ and }m\in M\} \subset G \times M\]
is therefore semi-algebraic and definably proper in
the sense of \cite[(2.13) p.166]{VDD}. Hence the quotient $G\to G/M$
exists semi-algebraically \cite[(2.15) p.166]{VDD}.
\end{proof}
%%%%%%%%%%%%%%%%%%%%%%%%%%%%%%%%%%%%%%%%%%%%%%%%%

\subsection{Siegel sets}  \label{Siegel sets}
A crucial ingredient in this paper in the classical notion of Siegel sets for
$\G$, which we recall now. We follow \cite[\textsection 2]{BJ} and refer to \cite[\S
12]{bor} for details.

Let $\bP$ be a $\QQ$-parabolic subgroup of $\G$. We denote by
$\bN_P$ its unipotent radical and by $\bL_P$ the Levi quotient $\bN_P
\backslash \bP$ of $\PP$. Let $N_P$, $P$, and $L_P$ be the Lie groups
of real points of $\bN_P$, $\bP$ and $\bL_P$ respectively. Let $\bS_P$
be the split center of $\bL_{P}$ and $A_P$ the connected component of
the identity in $\bS_{P}(\RR)$. Let $\bM_P := \cap_{\chi \in
  X^*(\bL_{P})} \ker  \chi^2$ and $M_P= \bM_P(\RR)$. Then $L_P$ admits
a decomposition $L_P = A_P M_P$. 

Let $X$ be the symmetric space of maximal compact subgroups of $G :=\G(\RR)^+$. Choosing a point $x \in X$ corresponds to choosing a maximal
compact subgroup $K_x$ of $G$, or equivalently a Cartan involution
$\theta_x$ of $G$. The choice of $x$ defines a unique real Levi subgroup $\bL_{P,x} \subset \bP_\RR$
lifting $(\bL_P)_\RR$ which is $\theta_x$-invariant, see
\cite[1.9]{BS}. Although $\bL_P$ is defined over $\QQ$ this is not
necessarily the case for $\bL_{P,x} $. The parabolic group $P$
decomposes as 
\begin{equation} \label{LD}
P= N_P A_{P, x} M_{P_{x}} \;\;,
\end{equation}
inducing a horospherical decomposition of $G$:
\begin{equation} \label{HD}
G = N_P A_{P, x} M_{P_{x}} K_x \;\;.
\end{equation}
We recall (see \cite[Lemma 2.3]{BJ} that the right action of $P$ on
itself under the horospherical decomposition is given by
\begin{equation} \label{mult}
(n_0 a_0 m_o) (n, a, m) = (n_0 \cdot (a_0 m_0) n (a_0 m_0)^{-1}, a_0
a, m_0m) \;\;.
\end{equation}

{\it In the following the reference to the basepoint $x$ in various
subscripts is omitted.}
We let $\Phi(A_P, N_P)$ be the set of characters of $A_P$ on the Lie
algebra $\Fn_P$ of $N_P$, ``the roots of $P$ with respect to
$A_P$''. The value of $\alpha \in \Phi(A_P, N_P)$ on $a \in A_P$ is
denoted $a^\alpha$. Notice that the map $a \mapsto a^\alpha$ from
$A_P$ to $\RR^*$ is semi-algebraic.

There is a unique subset $\Delta(A_P, N_P)$ of
$\Phi(A_P, N_P)$ consisting of $\dim A_P$ linearly independent roots,
such that any element of $\Phi(A_P, N_P)$ is a linear combination with
positive integral coefficients of elements of $\Delta(A_P, N_P)$ to be
called the simple roots of $P$ with respect to $A_P$. 

\begin{defi}(Siegel set) \label{defiSiegel}
For any $t >0$, we define $A_{P, t} = \{ a \in A_P \; | \; a^\alpha
>t, \; \alpha \in \Delta(A_P, N_P) \}$.
For any bounded sets $U \subset N_P$ and $W \subset M_{P}K$ the subset
$\FS:= U \times A_{P,t} \times W \subset G$ is called a
Siegel set for $G$ associated to $\bP$ and $x$.
\end{defi}

\noindent
When $U$ and $W$ are chosen to be relatively compact open semi-algebraic subsets of
$N_P$ and $M_P K$ respectively then the Siegel set $\FS= U \times A_{P, t}
\times W$ is semi-algebraic in $G$. {\em We will
only consider such semi-algebraic Siegel sets in the rest of the text.}

\medskip The following lemma follows immediately from \Cref{defiSiegel}:
\begin{lem}\label{lemma basepoint}
If $\FS$ is a Siegel set for $G$ associated to $\bP$ and $x$, then $g\FS g^{-1}$ is a
Siegel set for $G$ associated to $g \bP g^{-1}$ and $g x$.
\end{lem}

\begin{defi}
Let $M \subset G$ be a connected compact subgroup.  A Siegel set $\FS$
for the homogeneous space $G/M$ is a semi-algebraic subset of $G/M$ of the form
$\pi_x(\FS)$, where $\FS \subset G$ denotes a Siegel set for $G$
associated to some parabolic $\QQ$-subgroup $\bP 
\subset \G$, $K_x$ is a maximal compact subgroup containing a
conjugate $gMg^{-1}$ for some $g \in \G(\QQ)$ and $\pi_x: G \lo
G/gMg^{-1} \simeq G/M$.
\end{defi}

\begin{rem} \label{remarque}
  \begin{itemize}
    \item[(1)] It follows from the definition that for $\FS \subset G/M$ a Siegel set and $g \in \G(\QQ)$
      the translate $g \FS$ is a Siegel set of $G/M$.
    \item[(2)] As the projections $\pi_x$ are semi-algebraic and we
        consider only semi-algebraic Siegel sets in $G$, Siegel sets in $G/M$
        are semi-algebraic.
      \end{itemize}
    \end{rem}

%{\color{magenta}
%The space $G/M$ comes with a natural basepoint (the identity coset) and the spaces $G/M$ and $G/gMg^{-1}$ are nat%urally isomorphic as semialgebraic (left) homogeneous spaces via the right action by $g^{-1}$.  The notion of semi-alg%ebraic Siegel set of $G/M$ does not depend on the choice of basepoint by the following lemma (which we've implicitly %assumed in \Cref{reduction}(1)).

%\begin{lem}\label{lemma basepoint}  Let $\mathfrak{S}\subset G/M$ be a
%  semi-algebraic Siegel set associated to $\bP$ and $x\in X$ and let
%  $x'\in X$ be another choice of maximal compact
%  $K_{x'}=gK_{x}g^{-1}$.  Then $\mathfrak{S}g^{-1}\subset G/gMg^{-1}$
%  is a semi-algebraic Siegel set associated to $\bP$ and $x'$. 
%\end{lem}
%\begin{proof}As $G=PK$ we can write $g=pk$ so $x'=px$.  Write $p=nam$ with respect to the Langlands decompositi%on of $P$ at $x$.  The horospherical decomposition of $G$ at $x'$ is
%\[G=N_P\cdot pA_{P,x}p^{-1}\cdot pM_{P,x}p^{-1}\cdot pK_{x}p^{-1}\]
%and so if we have $\mathfrak{S}=UA_{P,t} W$ then
%\[\mathfrak{S}g^{-1}=UA_{P,t}Wk^{-1}p^{-1}=Un^{-1}\cdot pA_{P,t}p^{-1}\cdot pWk^{-1}p^{-1}\]
%as $pA_{P,t}p^{-1}=nA_{P,t}n^{-1}$.  Since $Un^{-1}$ and $pWk^{-1}p^{-1}$ are bounded and semi-algebraic if $U$ an%d $W$ are, the claim is proven. 
%\end{proof}}

\begin{prop} \cite[Prop. 2.5]{BJ} \label{reduction}
\begin{itemize}
\item[(1)] There are only finitely many $\Gamma$-conjugacy classes of
  parabolic $\QQ$-subgroups. Let $\bP_1$, \dots, $\bP_k$ be a set of
  representatives of the $\Gamma$-conjugacy classes of parabolic
  $\QQ$-subgroups. There exists Siegel sets $\FS_i:= U_i \times A_{P_i,t_i}
  \times W_i$ associated to $\bP_i$ and $x_i$, $1 \leq i \leq k$, whose
  images in $\Gamma \backslash G/M$ cover the whole space. 
\item[(2)] For any two parabolic subgroups $\bP_i$ and Siegel
  sets $\FS_i$ associated to $\bP_i$, $i=1, 2$, the set
  $$\Gamma_{\FS_1, \FS_2}:=\{ \gamma \in \Gamma \; |\; \gamma \FS_1 \cap \FS_{2} \not =
  \emptyset\}$$ is finite.
\item [(3)] Suppose that $\bP_1$ is not $\Gamma$-conjugate to
  $\bP_2$. Fix $U_i$, $W_i$, $i =1,2$. Then  $\gamma \FS_1\cap \FS_2=
  \emptyset$ for all $t_1, t_2$ sufficiently large.
\item[(4)] For any fixed $U, W$, when $t \gg0$, $\gamma \FS \cap \FS=\emptyset$ for all
  $\gamma \in \Gamma - \Gamma_P$, where $\Gamma_P := \Gamma \cap P$.
\item[(5)] For any two different parabolic subgroups $\bP_1$ and
  $\bP_2$, when $t_1, t_2 \gg 0$ then $\FS_1 \cap \FS_2 = \emptyset$.
  \end{itemize}
\end{prop}

\subsection{The Borel-Serre compactification $\ol{S_{\Ga, G,
      M}}^{BS}$} \label{borel serre}
In \cite{BS} Borel and Serre construct a natural compactification
$\overline{S_{\Gamma, G, K}}^{BS}$ of any arithmetic locally symmetric space
$S_{\Gamma, G, K}$ in the category of real-analytic manifolds with
corners, using the notion of geodesic actions and $S$-spaces. 
In \cite[\textsection 3]{BJ} Borel and Ji give a uniform construction
of the so-called Borel-Serre compactification $\overline{S_{\Ga, G,
    M}}^{BS}$ of any arithmetic quotient $S_{\Ga, G, M}$ 
in the category of real-analytic manifolds with corners, simplifying
the approach of \cite{BS} as they do not rely anymore on the notion of
$S$-spaces and delicate inductions: they construct a partial
compactification $\overline{G}^{BS}$ of $G$ in the category of
real-analytic manifolds with corners \cite[Prop.6.3]{BJ}, such that the left $\G(\QQ)^+$-action 
on $G$ (see \cite[prop. 3.12]{BJ}) and the commuting right $K$-action of a maximal
compact subgroup $K$ (see \cite[Prop.3.17]{BJ}) both extend to an action by weakly analytic maps to
$\overline{G}^{BS}$ (see proof of \cite[Prop. 6.4]{BJ}). For any neat
arithmetic subgroup $\Ga$ of $G$ and compact subgroup $M$ of $G$,  the
action of $\Gamma \times M$ on $\overline{G}^{BS}$ is free and proper. The quotient $\ol{S_{\Ga,
G,  M}}^{BS}:= \Gamma \backslash \ol{G}^{BS} /M$ provides a
compactification of the arithmetic quotient $S_{\Ga, G, M}$ in the
category of real-analytic manifolds with corners. 

Let us provide the details of this construction we will need.
Let $\bP \subset \G$ be a parabolic subgroup. Let $\Delta= \{ \alpha_1,
\ldots, \alpha_r\}$ be the set of simple roots in $\Phi(A_P,
N_P)$. Consider the semi-algebraic diffeomorphism $e_P: A_P \lo (\RR_{>0})^r$ defined
by  \begin{equation} \label{e_P}
e_P (a) = (a^{-\alpha_{1}}, \ldots, a^{-\alpha_{r}}) \in (\RR_{>0})^r
\subset \RR^r\;\;.
\end{equation}
Let $\ol{A_P} = [0, \infty)^r \subset \RR^r$ be the closure of
$e_P(A_P)$ in $\RR^r$. We denote by $\ol{A_{P,t}} \subset \ol{A_P}$ the closure of $e_P(A_{P,t})$. 

\medskip 
Let $$\ol{G}^{BS}= G \cup \coprod_{\bP \subset \G} (N_P \times (M_P
K))$$ be the Borel-Serre partial compactification of $G$ 
constructed in \cite[\textsection 3.2]{BJ}. The topology on
$\ol{G}^{BS}$ is such that an unbounded sequence $(y_j)_{j \in \NN}$ in
$G$ converges to a point $(n, m) \in N_P \times (M_P K)$
if and only if, in terms of the horospherical decomposition $G = N_P
\times A_P \times (M_PK)$, $y_j = (n_j, a_j, m_j)$ with $n_j \in N_P$,
$a_j \in A_P$, $m_j \in M_P K$, and the components $n_j$, $a_j$ and $m_j$
satisfy the conditions:

1) For any $\alpha \in \Phi(A_P, N_P)$, $(a_j)^\alpha \lo + \infty$,

2) $n_j \lo n$ in $N_P$ and $m_j \lo m$ in $M_P K$.

We refer to \cite[p274-275]{BJ} for the precise description of the
similar glueing between $N_P \times (M_P K)$ and $N_Q \times (M_Q K)$ for two
different parabolic subgroups $\bP \subset \bQ$.

Then:
\begin{prop} \cite[Prop.3.3]{BJ}
The embedding $N_P \times A_P \times (M_PK) = G \subset
\ol{G}^{BS}$ extends naturally to an embedding $N_P \times \ol{A_P}
\times (M_P K) \hookrightarrow \ol{G}^{BS}$.
\end{prop}
We denote by $G(P)$ the image of $N_P \times \ol{A_P}
\times (M_P K)$ under this embedding. It is called the corner
associated with $\bP$. As explained in \cite[Prop. 6.3]{BJ}
$\ol{G}^{BS}$ has the structure of a real-analytic manifold with
corners, a system of real analytic neighbourhood of a point $(n, m) \in N_P \times
(M_PK)$ being given by the $\ol{\FS_{U,t, W}}:= U \times \ol{A_{P, t}} \times W$, for $U$ 
a neighborhood of $n$ in $N_P$, $W$ a neighborhood of $m$ in $M_PK$
and $t>0$, see \cite[Lemma 3.10, Prop. 6.1 and Prop. 6.3]{BJ}. As the
right action of any compact subgroup $M$ of $K$ on $G$ extends to
a proper real analytic action on $\ol{G}^{BS}$, the quotient
$\ol{G}^{BS}/M$ is a partial compactification of $G/M$ in the category
of real-analytic manifolds with corners.

The left $\G(\QQ)$-multiplication on $G$ extends to a real
analytic action on $\ol{G}^{BS}/M$: see \cite[Prop. 3.12]{BJ} for the
extension to a continous action and the proof of \cite[Prop. 6.4]{BJ}
for the proof that the extended action is real analytic. The
restriction of this extended action
to a neat $\Ga$ is free and properly discontinuous (see
\cite[Prop. 3.13 and Prop. 6.4]{BJ}). Then $\ol{S_{\Ga, G, M}}^{BS}:=
\Gamma \backslash \ol{G}^{BS}/M$ is a compact real analytic 
manifold with corners compactifying $S_{\Ga, G, M}$. We denote by $\ol{\pi} : \ol{G}^{BS}/M \lo \ol{S_{\Ga, G, M}}^{BS} $
the extension of $\pi$.

\section{Proof of \Cref{definability}}

\subsection{Proof of \Cref{definability}(1)} \label{atlas} 

\begin{proof}[\unskip\nopunct]

By \Cref{reduction}(1) there exist finitely many $\bP_1, \ldots, \bP_k$
parabolic subgroups of $\G$ and Siegel sets $\FS_i:= U_i \times A_{P_{i}, t_{i}}
  \times W_i$, $1 \leq i \leq k$, with $U_i$, $W_i$ compact semi-algebraic subsets of
$N_{P_{i}}$ and $M_{P_{i}} K/M$ respectively, whose images $V_i:=
\pi(\FS_{i})$, $1 \leq i \leq k$ cover $S_{\Ga, G, M}$. 
The real analytic manifold
$S_{\Ga, G, M}$ can thus be obtained as the quotient of 
$\coprod_{i=1}^k \FS_i$ by the \'etale equivalence relation $E$ defined by 
$$ x_1 \in \FS_{i_{1}} \sim_{E} x_2 \in \FS_{i_{2}} \iff \exists \,
\gamma \in \Gamma
\;| \; \gamma x_1 = x_2 \;\;.$$ The $V_i$'s, $1 \leq i \leq k$,
provide a cover of $S_{\Ga, G, M}$ by open real-analytic charts.

For $1 \leq i \leq k$ let $\cl(\FS_i)$ denotes the topological closure
of $\FS_i$ in $G/M$. The quotient $S_{\Ga, G, M}$ also identifies with
the quotient of $\coprod_{i=1}^k \cl(\FS_i)$ by the same equivalence
relation $\sim_{E}$. 
As each $\FS_i$, $1 \leq i \leq k$, is semi-algebraic, their closure
$\cl(F_i)$ too, hence the set $\coprod_{i=1}^k \cl(\FS_i)$ is $\RR_\alg$-definable. As the action of
$\Ga$ is real-algebraic on $G$, the equivalence relation $\sim_E$ on $\coprod_{i=1}^k \cl(\FS_i)$ is
$\RR_\alg$-definably proper by \Cref{reduction}(2) in the sense of
\cite[(2.13) p.166]{VDD}. By \cite[(2.15) p.166]{VDD} the 
quotient $S_{\Ga, G, M}= (\coprod_{i=1}^k \cl(\FS_i))/\sim_{E}$ is
naturally an $\RR_\alg$-definable manifold: each $\cl(V_i):= \pi(\cl(\F_i))$
is $\RR_{\alg}$-definable and the restriction $\pi_i: \cl(\FS_i) \lo \cl(V_i)$
of $\pi$ to $\FS_i$ is $\RR_{\alg}$-definable. As each $\FS_i$ is
semi-algebraic, it follows that 
$V_i:= \pi_i(\FS_i)$, $1 \leq i \leq k$, is semi-algebraic and the
$V_i$'s, $1\leq i \leq k$, form an explicit finite open atlas of the
$\RR_\alg$-definable manifold $S_{\Ga, G, M}$.

Let $\FS \subset G/M$ be any Siegel set. By \Cref{reduction}(1) there
exists $\gamma \in \Ga$ such that $\gamma \FS$ is associated to
one of the parabolics $\bP_i$ for some $1 \leq i \leq k$. Replacing $\FS_i$ by a bigger Siegel set
for $\bP_i$ if necessary in the previous construction, we can assume without loss of generality that
$\gamma\FS$ is contained in $\FS_i$. Hence $\pi_{|\FS}: \FS
\lo S_{\Ga, G, M}$ coincide with the composite 
$$\FS \stackrel{\gamma \cdot}{\rightarrow} \gamma\FS \hookrightarrow \FS_i
\stackrel{\pi_i}{\rightarrow} S_{\Ga, G, M}$$ 
hence is $\RR_\alg$-definable.

With the notations above, the set $\cF:=
\cup_{i=1}^k \FS_i \subset G/M$ is a semi-algebraic fundamental set for the action
of $\Ga$ on $G/M$. As each $\pi_i: \FS_i \lo S_{\Ga, G, M} $ is
$\RR_\alg$-definable, it follows that $\pi_{\cF}: \cF \lo S_{\Ga, G, M}$ is
$\RR_{\alg}$-definable.

By \Cref{ManifoldWC} the compact real analytic manifold with corners
$\ol{S_{\Ga, G, M}}^{BS}$ admits a natural structure of
$\RR_\an$-definable manifold with corner. Explicitly: the images $\ol{V_i} :=
\ol{\pi}(\ol{\FS_{i}})$ of $\ol{\FS_{i}}:= U_i \times \ol{A_{P_{i}, t_{i}}}
\times W_i \subset \ol{\G}^{BS}$, $1 \leq i \leq k$, cover
$\ol{S_{\Ga, G, M}}^{BS}$ and form a finite atlas of the
$\RR_\an$-definable manifold with corners $\ol{S_{\Ga, G,
    M}}^{BS}$. Let us show that the $\RR_\an$-definable
manifold structure on $S_{\Ga, G, M}$ obtained by restriction of this structure of
$\RR_\an$-definable manifold with corner on $\ol{S_{\Ga, G, M}}^{BS}$
coincide with the $\RR_\an$-definable
manifold structure extending the $\RR_\alg$-definable
manifold structure on $S_{\Ga, G, M}$ we just constructed.
We are reduced to showing that for each $i$, $1 \leq i \leq k$, the
map $\FS_i \stackrel{\pi_i}{\lo} V_i \hookrightarrow 
\ol{V_i}$ is $\RR_{\an}$-definable.  
It factorises as 
$$
\xymatrix@1{
\FS_i  \ar[rr]^{1_{U_{i}} \times e_{P_{i}}
    \times 1_{W_{i}}} & &\ol{\FS_{i}}\ar[rr]^{\ol{\pi_i}} & &\ol{V_i}\;\;.
}
$$
On the one hand, it follows from the definition~(\ref{e_P}) of
$e_{P_{i}}: A_{P_{i}, t_{i}}  \lo \ol{A_{P_{i}, t_{i}}} \subset \RR^{r_{i}}$ that $e_{P_{i}}$, hence also $1_{U_{i}} \times e_{P_{i}}
    \times 1_{W_{i}}$, is semi-algebraic. On the other hand, the map $\ol{\pi_i}: U_i \times
\ol{A_{P_{i}, t_{i}}} \times W_i \lo \ol{V_i}$ is a real analytic map
between compact sets, hence is $\RR_\an$-definable. This concludes the
proof that $\pi_{|\FS}: \FS \lo S_{\Ga, G, M}$ is
$\RR_\an$-definable.

This concludes the proof of \Cref{definability}(1).
\end{proof}

\begin{rem}
Although the $\ol{\FS_i}$, $1 \leq i \leq k$, are semi-algebraic, it
is not true in general that $\ol{S_{\Ga, G, M}}^{BS}$ admits a natural
structure of $\RR_\alg$-definable manifolds with corners: if $\bP
\subset \bQ$ are two parabolics of $\G$ it follows from the proof of
\cite[Prop. 6.2]{BJ} that the inclusion of the corner
$G(Q) \subset G(P)$ is real-analytic but not semi-algebraic in general.
\end{rem}

\subsection{Morphisms of arithmetic quotients are definable: proof of
  \Cref{definability}(2)}
  
\begin{proof}[\unskip\nopunct]  

Notice that the statement of \Cref{definability}(2) is
non-trivial even in $\RR_\an$. For instance in the case
where $f: \G' \lo \G$ is a strict inclusion the morphism $f$ does not usually extend
to a real analytic morphism (or even a continuous one)
$\ol{f}$ between the Borel-Serre compactifications (in other words the
Borel-Serre compactification is not functorial). The problem is
that two parabolic subgroups $\bP_i \subset \G$, $i=1,2$, can be non conjugate
under $\Ga$ while their intersections $\bP_i \cap \G'$ are
$\Ga'$-conjugate parabolic subgroups of $\G'$. However
\Cref{definability}(2) will follow from a finiteness result for Siegel
sets due to Orr (see \Cref{orr}).

First, we claim that for $g \in \G(\QQ)$ the morphism of arithmetic
quotients $(\inter(g), g): S_{g^{-1}\Gamma g , G, M} \lo S_{\Gamma ,G,
   M}$ (the left multiplication by $g$) is $\RR_\alg$-definable: 
  this follows immediately from \Cref{remarque}(1).

Let $(f,g): S_{\Gamma', G', M'} \lo S_{\Gamma,G, M}$ be a general morphism of
arithmetic quotients.  As $(f, g)= (\inter(g^{-1}) \circ f, 1) \circ (
\inter(g), g)$ we are reduced to considering morphism of
arithmetic quotients $f:=(f,1): S_{\Gamma', G', M'}\lo S_{\Gamma, G,
  M}$ deduced from a morphism $f: \G' \lo \G$
of semi-simple linear algebraic $\QQ$-group such that $f(M') \subset
M$ and $f(\Ga') \subset \Ga$.

Let $(V'_i)_{1\leq i \leq k}$ be an
$\RR_{\alg}$-atlas for $S_{\Ga', G', M'}$ as in
\Cref{atlas}. Showing that $f: S_{\Ga', G', M'} \lo S_{\Ga, G, M}$ is
$\RR_{\alg}$-definable is equivalent to showing that for each
$i$, $1 \leq i \leq k$, the restriction $f: V'_i \lo S_{\Ga, G, M}$ is $\RR_{\alg}$-definable.
As the diagram
$$
\xymatrix{
\FS_i':= U'_{i} \times A'_{P'_{i}, t'_{i}} \times W'_{i} \ar[r]^>>>>>>f
\ar[d]_{\pi'_{i}} & G \ar[d]^{\pi} \\
V'_{i} \ar[r]_f  & S_{\Ga, G, M}
}
$$
is commutative, it is enough to show that the composite
\begin{equation}  \label{composite}
\FS'_i
\stackrel{f}{\longrightarrow} G \stackrel{\pi}{\longrightarrow}
S_{\Ga, G, M}
\end{equation}
is $\RR_{\alg}$-definable.

The case where $\G' = \G$ is clear from \Cref{definability}(1) (notice that in that case Goresky
and MacPherson show in \cite[Lemma 6.3]{GMcP} (and its proof)
that $f$ extends uniquely to a real analytic morphism $\ol{f} :  \ol{S_{\Gamma', G', M'} }^{BS} \lo \ol{S_{\Gamma,G,
    M}}^{BS}$).

Suppose that $f: \G' \lo \G$ is surjective. Without loss of generality
we can assume that $\G$, and then $\G'$, are adjoint. Then $\G' = \G
\times \bH$, $S_{\Ga', G', M'} = S_{\Ga' \cap G, G, M' \cap G} \times S_{\Ga' \cap H, H, M' \cap H}$, the map $f$ coincides with the
projection onto the first factor (this projection is obviously
$\RR_\alg$-definable) composed with the morphism of
arithmetic quotients $i: S_{\Ga' \cap G,
  G, M' \cap G}  \lo S_{\Ga, G, M}$, which is $\RR_\alg$-definable from the
case $\G' = \G$. This proves \Cref{definability}(2) when $f: \G' \lo \G$ is surjective.

We are thus reduced to proving \Cref{definability}(2) in the case
where $f: \G' \lo \G$ is a strict inclusion. We use the following:
\begin{theor} (\cite[Theor.1.2]{Orr}) \label{orr}
Let $\G$ and $\bH$ be reductive linear $\QQ$-algebraic groups, with
$\bH \subset \G$. Let $\FS_H:= U_H \times A_{P_{H}, t} \times W_H \subset \HH(\RR)$ be a Siegel set
for $\bH$. 

Then there exists a finite set $C \subset \G(\QQ)$ and a Siegel set
$\FS:= U \times A_{P, t} \times W \subset \G(\RR)$ such that $\FS_H \subset C\cdot \FS$.
\end{theor}

Applying this result to $\G' \subset \G$ and the Siegel set $\FS'_i$
of $\G'$, there exists a finite set $C_i \subset \G(\QQ)$
and a Siegel set $\FS_i:= U_i \times A_{P_{i}, t_{i}} \times W_i$ such that the
composition~(\ref{composite}) factorizes as
\begin{equation} \label{factor}
\FS'_i \lo C_i \cdot \FS_i \stackrel{\pi}{\longrightarrow}
S_{\Ga, G, M}\;\;.
\end{equation}

The inclusion $\FS'_i \lo C_i \cdot \FS_i$ is semi-algebraic. The map
$C_i \cdot \FS_i \stackrel{\pi}{\longrightarrow}
S_{\Ga, G, M}$ is $\RR_{\alg}$-definable by
\Cref{definability}(1). 

This finishes the proof of \Cref{definability}(2).
\end{proof}

%%%%%%%%%%%%%%%%%%%%%%%%%%%%%%%%%%%%%%%%%%%%%%%%%%%%%
\section{Definability of the period map}

\subsection{Reduction of \Cref{definability period map}
  to a local statement}

In the situation of \Cref{definability period map}, 
let $S \subset \ol{S}$ be a smooth compactification such that $\ol{S}
-S$ is a normal
crossing divisor. Let $(\overline{S_i})_{1\leq i \leq m}$ be a finite
open cover of $\ol{S}$ such that the pair $(\ol{S_{i}}, S_i := S\cap
\overline{S_i})$ is biholomorphic to $(\Delta^n, (\Delta^*)^{r_{i}}
\times \Delta^{l_i:= n-r_{i}})$. To show that the period map $\Phi_S: S \lo \Hod^0(S,
\VV) = S_{\Ga, G, M}$ is $\RR_{\an, \exp}$-definable, it is enough to
show that for each $i$, $1 \leq i \leq m$, the restricted period map
\begin{equation} 
{\Phi_S}_{|S_{i}}: S_i= (\Delta^*)^{r_{i}} \times \Delta^{l_{i}} \lo S_{\Ga, G,
  M}
\end{equation}
is $\RR_{\an, \exp}$-definable. Without loss of generality we can
assume that $r_i= n$ and $l_i=0$ by allowing some factors with trivial
monodromies. Finally we are reduced to proving:

\begin{theor} \label{polydisc}
Let $\VV $ be a polarized variation of pure
Hodge structures of weight $k$
over the punctured polydisk $(\Delta^*)^n$, with period map $\Phi:
(\Delta^*)^n \lo S_{\Ga, G, M}$.
Then $\Phi$ is $\RR_{\an, \exp}$-definable.
\end{theor}

\subsection{Proof of \Cref{polydisc} assuming
  \Cref{finitely many Siegel}}

\begin{proof}[\unskip\nopunct]
Let us fix $x_0$ a basepoint in  $(\Delta^*)^n$. We denote by $V_\ZZ$ the fiber
$\VV_{x_{0}}$ of $\VV$ at $x_0$ (modulo torsion) and $V_\QQ= V_\ZZ
\otimes_\ZZ \QQ$.  We further denote by $Q_\ZZ$ the polarization form on $V_\ZZ$, and $Q_\CC$ its $\CC$-linear extension to $V_\CC$.

It follows from Borel's monodromy
theorem \cite[Lemma (4.5)]{Schmid} that the monodromy transformation $T_i \in
\G(\ZZ) \subset \GL(V_\ZZ)$, $1\leq i \leq n$, of the local system $\VV$, corresponding to counterclockwise simple
circuits around the various punctures, are
quasi-unipotent. Replacing $(\Delta^*)^n$ by a
finite \'etale cover if necessary we can assume without loss of
generality that all the $T_i$'s are unipotent. Let $N_i \in
\Fg_\QQ$, $1 \leq i \leq n$, be the logarithm of $T_i$; this is a nilpotent element in the (rational) Lie algebra $\Fg_\QQ$ of $\G$.

Let $\FH$ denote the Poincar\'e upper-half plane and $e(\; \cdot\;):=\exp(2 \pi i\;
\cdot\;): \FH \lo \Delta^*$ the uniformizing map of $\Delta^*$. Let $\FS_\FH
\subset \FH$ be the usual Siegel fundamental set $\{ (x, y)
\in \FH\; | \; 0<x < 1, \; 1 <y \}$.
Consider the commutative diagram
$$ 
\xymatrix{
\FS_\FH^{n}  \subset \FH^{n}
\ar[d]_{p= e(\cdot\;)^{n}}
\ar[r]^>>>>>>{\tilde{\Phi}} & D= G/M \ar[d]^{\pi} \\
(\Delta^*)^{n} \ar[r]_{\Phi} & S_{\Ga, G,M}\;\;, 
}
$$
where $\tilde{\Phi}$
is the lifting of $\Phi$ to the universal cover.
As the restriction $\exp(2 \pi i \;\cdot\;)_{|\FS_{\FH}}$ is $\RR_{\an,
  \exp}$-definable, the map $p_{|\FS_\FH^{n}}: {\FS_\FH}^{n} \lo
(\Delta^*)^{n}$ is $\RR_{\an,
  \exp}$-definable. We are reduced to proving that the composition
$
\xymatrix@1{{\FS_\FH}^{n} \ar[r]^{\ti{\Phi}} & G/M \ar[r]^\pi
&S_{\Ga, G,M}}$ is $\RR_{\an, \exp}$-definable.

\begin{lem} \label{l1}
The map $\tilde{\Phi} : {\FS_\FH}^{n}\lo G/M$
is $\RR_{\an, \exp}$-definable.
\end{lem}

\begin{proof}
The nilpotent orbit theorem \cite[(4.12)]{Schmid} states that
(after maybe shrinking the polydisk) the map
$\tilde{\Phi}: {\FS_\FH}^{n}  \lo G/M$ is of
the form $$\tilde{\Phi} (z) = \exp \left(\sum_{j=1}^{n}  z_j N_j\right) \cdot \Psi(
p( z))$$
for $\Psi: \Delta^n \lo \check{D}$ a
holomorphic map and $\check{D} \supset D$ the compact dual of $D$. The
map $\Psi$ is the restriction to a relatively compact set of a real
analytic map. As $p_{|{\FS_\FH}^{n}}: {\FS_\FH}^{n} \lo
\Delta^n$ is $\RR_{\an, \exp}$-definable, it follows that $(z \mapsto
\Psi(p(z))$ is $\RR_{\an, \exp}$-definable.

The action of $\G(\CC)$ on the compact dual $\check{D}$ is algebraic, hence $\RR_{\an, \exp}$-definable; $D$ is a
semi-algebraic subset of $\check{D}$ and $\exp (\sum_{j=1}^n  z_j
N_j)$ is a polynomial in the $z_j$'s as the $N_j$'s are nilpotent. 

Hence the result.
\end{proof}

\begin{rem}
Notice that \Cref{l1} appears also in \cite[Lemma 3.1]{BaT}.
\end{rem}

It moreover follows from \Cref{finitely many Siegel}, proven below, that there exist finitely many Siegel sets $\FS_i$ $(i \in I)$
for $G/M$ such that $\tilde{\Phi}({\FS_\FH}^{n}) \subset \bigcup_{i
  \in I}\FS_i$. As $\pi_{|\FS_i}: \FS_i \lo
S_{\Ga, G, M}$, $i \in I$, is $\RR_{\an}$-definable by 
\Cref{definability}(2), and the set $I$ is finite, we deduce from \Cref{l1} that 
$ \pi \circ \tilde{\Phi} : {\FS_\FH}^{n}
\lo S_{\Ga, G,M}$
is $\RR_{\an, \exp}$-definable. This concludes the proof of
\Cref{polydisc}, hence of
\Cref{definability period map}, assuming \Cref{finitely many Siegel}.
\end{proof}

\subsection{Roughly polynomial functions}
Before the proof of \Cref{finitely many Siegel} we discuss a class of
functions with the same asymptotics as the Hodge form.

\begin{defi}\label{rough}Let $\Sigma_n\subset\FS_\FH^n$ be the region
  $0< x_i< 1 $ and $y_1\geq\cdots \geq y_n> 1$ for
  $x_i=\mathrm{Re}\, z_i$ and $y_i=\mathrm{Im}\, z_i$.  Let $\cO$ be
  the ring of real restricted analytic functions on $\Delta^n$ pulled
  back to $\Sigma_n$ via $p:\FH^n\to(\Delta^*)^n$, $\cO[x,y,y^{-1}]$
  the ring of polynomials in
  $x_1,\ldots,x_n,y_1,\ldots,y_n,y_1^{-1},\dots,y_n^{-1}$ with
  coefficients in $\cO$, and $\cO(x,y)$ its fraction field.  By a
  \emph{monomial} we mean an element of $\cO[x,y,y^{-1}]$ of the form
  $y_1^{s_1}\cdots y_n^{s_n}$ for integers $s_i\in\Z$.  We say a
  function $f\in \cO(x,y)$ is \emph{roughly monomial} if it is within
  a multiplicatively bounded constant of a monomial on $\Sigma_n$.  We
  say that $f\in\cO(x,y)$ is \emph{roughly polynomial} if it is of the
  form $\frac{g}{h}$ where $g\in\cO[x,y,y^{-1}]$ and $h$ is roughly
  monomial.  Note that roughly polynomial functions form a ring which
  we denote $\mathcal{T}_n$. 
\end{defi}

In the following we write $f\ll g$ to mean $f<Cg$ for a constant $C>0$ and $f\sim g$ to mean $f\ll g$ and $g\ll f$.  Thus, $f\in\cO(x,y)$ is roughly monomial if $f\sim y_1^{s_1}\cdots y_n^{s_n}$.  The next lemma will allow us to understand the asymptotics of roughly polynomial functions by restricting to curves.

\begin{lem}\label{check on curves}
Let $f,g\in\cO(x,y)$ with $f$ roughly polynomial and $g$ roughly monomial. Assume that $|f|\ll |g|$ when restricted to any set of the 
form $$\Sigma_n \cap
\{\alpha_1z_1+\beta_1=\alpha_2z_2+\beta_2=\cdots=\alpha_{n_0}z_{n_0}+\beta_{n_0},
z_{n_0+1}=\zeta_{n_0+1},\dots,z_n=\zeta_n\}$$ for some $1\leq n_0\leq n$, $\zeta_{n_0+1},\dots,\zeta_n\in \FH$, $\alpha_1,\dots,\alpha_{n_0}\in \QQ_+^{n_0}$, and $\beta_1,\ldots,\beta_{n_0}\in \RR$. Then
$|f|\ll |g|$ on all of $\Sigma_n$.
\end{lem}

\begin{proof} 

By clearing denominators it is clearly sufficient to handle the case where $g$ is a monomial, and in fact where $g=1$ since $y_i$ is a unit in $\cO[x,y,y^{-1}]$. We proceed by induction on
$n$, with the case $n=1$ being immediate from the assumptions.
Separating out the powers of $x_1$ and $y_1$ we may write $f=\sum_j
a_j x_1^{j_1}y_1^{j_2}$, where the sum is over finitely many pairs
$j=(j_1,j_2)$ in $\ZZ^2$. 
Now, the $a_j$ are real analytic functions in $t_1:= e(z_1), x_2,y_2,y_2^{-1},t_2:=e(z_2),\dots,x_n,y_n,y_n^{-1},t_n:=e(z_n)$ up to an error on $\Sigma_n$ of $O(y_1^{A}e^{-2\pi y_1})$ for some positive $A>0$. Since this 
decays faster than any monomial, we may restrict to the case where the
$a_j$ are independant of $t_1$.  

The lemma will follow immediately once we prove the following claim:

\medskip

\noindent\textbf{Claim:} For each $j$ we have that $|a_j|y_1^{j_2} \ll 1$ on $\Sigma_n$. 

\medskip

First, we claim that all the powers $j_2$ of $y_1$ in $f$ are
non-positive. If this is not the case, we can fix the other variables
at a point where the coefficient of a positive power of $y_1$ is
non-zero, and get a contradiction as $y_1\to\infty$.  

Now, since the powers of $y_1$ are non-positive and $y_1\geq y_2,$ it is sufficient to prove the claim when we restrict to $y_1=Cy_2$ for any $C>0$. Consider for each positive integer $m$ and real number $c$ 
the function $f_{m,c}$ which we obtain from $f$ by setting
$z_1=mz_2+c$. Note that the assumptions of the lemma still apply to
$f_{m,c}$, and it follows by induction on $n$ that $f_{m,c}\ll 1$ on all of
$\Sigma_{n-1}$ (taken with respect to the ordered coordinates
$z_2,\ldots,z_n$).  
Thus, we have that $\sum_j a_j (mx_2+c)^{j_1}(my_2)^{j_2} \ll1 $ on all of $\Sigma_{n-1}$.  

Let $j_{\mathrm{max}}=(r_1,r_2)$ be the lexicographically maximal $j$ that occurs among all $j$ with $a_j\neq 0$.  Let $F_{m}:=\frac{1}{r_1!}\sum_{i=0}^{r_1} \binom{r_1}{i}(-1)^i f_{m,i}$. Note that $F_{m}\ll 1$ on $\Sigma_{n-1}$ and also that $F_{m} = \sum_{j=(r_1,j_2)} a_Jm^{j_2} y_2^{j_2}$. 
By taking a finite (constant) linear combination of the $F_{m}$ with distinct values of $m$ we may isolate the $a_{j_{\mathrm{max}}}y_2^{r_2}$ term, from which it follows that 
$a_{j_{\mathrm{max}}}y_2^{r_2}\ll 1$, thus proving the claim for this monomial. Subtracting it off and proceeding inductively, the claim follows.

\end{proof}

\subsection{The Hodge form is roughly polynomial}

For $u,v\in V_\CC$ we denote by $h(u,v): \Sigma_n\to\CC$ the function
mapping $z\in\Sigma_n$ to the Hodge form $h_z(u,v):=Q_\CC(C_z u,v)$ of
$u$ and $v$ at $\tilde\Phi(z)$, where $C_z$ is the Weil operator at
$\tilde{\Phi}(z)$.  Likewise we denote $h_z(u)=h_z(u,u)$.  The main result of this section is the following:

\begin{prop}\label{roughly two} For any $u,v\in V_\CC$, the function $h(u,v)$ is
roughly polynomial. 
\end{prop}

Let us introduce the notations needed for proving \Cref{roughly two}.  Let $C\subset \g_\RR$ be the convex cone generated by the monodromy logarithms
$N_i$.  Recall from \cite[p.468]{CKS} that for each $M\in C$ we have a
weight filtration $W(M)$ on $V_\ZZ$ (centered at 0).  Let $M_j=\sum_{i=1}^j N_i$ and $F=\Psi(0)\in \check D$.  We write $\tilde\Phi(z)=\gamma(z)F$
where $\gamma:\FH^n\to \G(\CC)$ is a lift of the form $\gamma(z)=e^{z\cdot N}g(p(z))$ for $g:\Delta^n\to \G(\C)$ a holomorphic function taken as follows.  Writing $\g^{p,q} $ for the Deligne splitting of the mixed Hodge structure on $\g_\RR$ induced by $(F,W(M_n))$, there is a unique holomorphic lift $v(t)\in \bigoplus_{p<0}\g^{p,q}$ with $\Psi(t)=e^{v(t)}F$.

Recall that there is a splitting (see \cite[Lemma 2.4.1]{Ka}) 
$V_\CC=\bigoplus_{p,q_1,\ldots,q_r}I^{p,q_1,\ldots,q_n}$
such that
$F^s=\bigoplus_{p\geq s}I^{p,q_1,\ldots,q_n}$ and $W(M_j)_s=\bigoplus_{p+q_j\leq s}I^{p,q_1,\ldots,q_n}$.
For simplicity we denote $\alpha=(p,q_1,\ldots,q_n)$.  There is also a rational splitting $V_\QQ=\bigoplus_{s_1,\ldots,s_n}J^{s_1,\ldots,s_n}$ such that $W(M_j)_s=\bigoplus_{s_j\leq s}J^{s_1,\ldots,s_n}$, and again for simplicity we denote $\sigma=(s_1,\ldots,s_n)$.

\Cref{roughly two} follows immediately from:

\begin{lem}\label{roughly}  Let $u\in I^\alpha$ and $v\in I^\beta$.  
\begin{enumerate}
\item
  $h(u)$ is roughly monomial.
\item $h(\gamma(z)u)$ is roughly monomial.
\item $h(u,v)$ is roughly polynomial.
\end{enumerate}
\end{lem}

\begin{proof}[Proof of \Cref{roughly}] 
The asymptotics of Hodge forms are well-studied, and we state the
precise result we will need: 

\begin{theor}[Theorem 5.21 of \cite{CKS} or Theorems 3.4.1 and 3.4.2 of \cite{Ka}]\label{asymptotics}  Let $u\in J^{s_1,\ldots,s_n}_\CC$.  Then on $\Sigma_n$ we have
\begin{enumerate}
\item $h(u)\sim (y_1/y_2)^{s_1}\cdots (y_{n-1}/y_n)^{s_{n-1}}y_n^{s_n}$;
\item $h\left(e^{z\cdot N}u\right)\sim (y_1/y_2)^{s_1}\cdots (y_{n-1}/y_n)^{s_{n-1}}y_n^{s_n}$.
\end{enumerate}
\end{theor}  

\begin{rem}We remark that Kashiwara's proof of \Cref{asymptotics} does
  not require the use of the $SL_2^n$-orbit theorem of \cite{CKS}. 
\end{rem}
 
Given \Cref{asymptotics}(1), proving part (1) of \Cref{roughly two} reduces to showing that $h(u)$ is in $\cO(x,y)$.  Choose a basis $w_i$ of $V_\CC$ such that each $w_i\in I^{\alpha_i}$ for some $\alpha_i=(p_i,q_1^i,\ldots, q_n^i)$ and the sequence $p_i $ is non-increasing.  Define an increasing filtration $K^\bullet$ as follows:  $K^j$ is the span of $w_1,\ldots, w_j$.  Note that $K^\bullet$ is a full flag refining the filtration $F^{-\bullet}$.   Define $B(u,v):=Q_\CC(u,\bar v)$ and for simplicity call $\gamma=\gamma(z)$.  An
$h$-orthogonal basis of the Hodge filtration at $\gamma F$ is obtained
from the $\gamma w_i$ by the Gram-Schmidt procedure (with respect to
$B$); call this basis $\tilde w_i$.  Let $u=\sum \tilde u_i$ with
$\tilde u_i$ a multiple of $\tilde w_i$ and likewise for $v$.  We then
have, denoting $w_{\det K^i}=w_1\wedge\cdots\wedge w_i$ and extending
the use of the symbols $B$ and $h$ to the corresponding hermitian forms on any wedge power (or
other tensor operation) of $V_\CC$:
\begin{equation}B(\tilde w_i)=\frac{B\left(\gamma w_{\det
        K^{i}}\right)}{B(\gamma w_{\det
      K^{i-1}})}\label{eq1}\end{equation} 
We 
also have 
\begin{equation}B(u,\tilde w_i)=\frac{B\left((\gamma w_{\det
        K^{i-1}})\wedge u,\gamma w_{\det K^{i}}\right)}{B(\gamma
    w_{\det K^{i-1}})}\;,\label{eq11}\end{equation} 
\begin{equation}h(\tilde u_i,\tilde v_i)=i^{2p_i-k} B(\tilde u_i,\tilde v_i)=i^{2p_i-k} \frac{B(u,\tilde w_i)B(\tilde w_i,v)}{B(\tilde w_i)}\;.\label{eq111}\end{equation}
Now, \eqref{eq1} and \eqref{eq11} are both in $\cO(x,y)$, so
$h(u,v)=\sum_i h(\tilde u_i,\tilde v_i)$ is in $\cO(x,y)$ as well. 
 
That $h(\gamma u)$ is in $\cO(x,y)$ similarly follows from the fact that the $B$-norm of the projection of $\gamma u$ to $\tilde w_i$ is likewise computed via
\begin{equation}B(\gamma u,\tilde w_i)=\frac{B(\gamma(w_{\det
      K^{i-1}}\wedge u),\gamma w_{\det K^{i}})}{B(\gamma w_{\det
      K^{i-1}})}.\end{equation} 
To finish the proof of part (2), we need the following lemma, which is also proven in \cite[Theorem 5.21]{CKS}.
\begin{lem}  On $\Sigma_n$ we have $h(\gamma u)\sim h\left(e^{z\cdot N}u\right)$.
\end{lem}
\begin{proof}Recall that $\gamma(z)=e^{z\cdot N}g(p(z))$ and $g(t)=e^{v(t)}$.  Griffiths transversality requires
$$e^{-\ad(v(t))}N_i+t_i\frac{1-e^{-\ad(v(t))}}{\ad(v(t))}\frac{\partial v}{\partial t_i}\in F^{-1}\g.$$
Since $N_i\in \g^{-1,-1}$ and $v(t)\in\bigoplus_{p<0}\g^{p,q}$, this implies that 
$$v(t)=\sum_{i=1}^n t_iv_i(t)$$
where $v_i(t)$ is a holomorphic function of $t_i,\ldots,t_n$ and
$\ad(N_j)v_i(t)=0$ for $j<i$.  Thus, $v_i(t)$ preserves each $W(M_j)$
for $j<i$.  It follows likewise that $g(t)=1+\sum_i t_i g_i(t)$ with
$g_i(t)\in \End(V_\C)$ preserving $W(M_j)$ for $j<i$.  Thus by
\Cref{asymptotics}, we have on $\Sigma_n$ 
$$h(e^{z\cdot N}g_i(t)u)\ll (y_1/y_2)^{s_1}\cdots (y_{i-1}/y_{i})^{s_{i-1}}(y_{i}/y_{i+1})^{s'_{i}}\cdots y_n^{s'_n} $$
for some $s'_{i},\ldots,s_n'$.  For any monomial $y^I$ only in
$y_i,\ldots, y_n$ and any $\epsilon>0$, we have that
$|t_i|^2y^I<\epsilon$ for sufficiently large $y_n$ and so the same is
true of $h(\gamma u-e^{z\cdot N}u)/h(e^{z\cdot N}u)$, whence the
claim. 
\end{proof}

Finally, $h(\gamma w_{\det K^i})=|B(\gamma w_{\det K^i})|$ is roughly
monomial by part (2).  It then follows that the same is true for
\eqref{eq1}, and thus that $h(\tilde u_i,\tilde v_i)$ is roughly
polynomial, as the numerator in \eqref{eq111} is clearly in
$\cO[x,y,y^{-1}]$. 
\end{proof}

\subsection{Proof of \Cref{finitely many Siegel}}

\begin{proof}
First, observe that the set $\FS_\FH^b$ is covered by finitely many sets
of the form $\Sigma_n$ (corresponding to the finitely many possible
orderings of the coordinates). Hence it is enough to show that
$\tilde\Phi(\Sigma_n)$ is contained in finitely many Siegel sets of
$D=G/M$.

The natural embedding $\G(\QQ)\subset \SL(V_\QQ)$ defines a natural
map $\iota:D\to X$, given by $\iota(x):= h_x$, where $X$ denotes the
symmetric space of positive definite symmetric forms on $V_\RR$. By
\cite[7.5]{BHC} the preimage of any Siegel set $\mathfrak{S}\subset X$
is contained in the union of finitely many Siegel sets of $D$.  It is thus enough to show that
$\iota\circ\tilde\Phi(\Sigma_n)$ is contained in finitely many Siegel
sets of $X$.

Siegel sets in $X$ can be understood in terms of reducedness of
positive definite forms with respect to a basis:
\begin{defi}\label{def reduced}  Given an integral (ordered) basis
  $e=\{e_i\}$ of $V_\ZZ$, a constant $C>0$, and a positive definite symmetric form $b$,
  we say $b$ is $(e,C)$-reduced if: 
\begin{enumerate}
\item $|b(e_i,e_j)| <Cb(e_i)$ for all $i,j$;
\item $b(e_i)<C b(e_j)$ for $i<j$;
\item $\prod_i b(e_i)<C \det(b)$.
\end{enumerate}
\end{defi}

\noindent
Given an integral (ordered) basis $e=\{e_i\}$ of $V_\ZZ$ and $C>0$
we define the subset $\mathfrak{T}_{e,C}=\{b\in X\mid b\mbox{ is
}(e,C)\mbox{-reduced} \}$.  By classical reduction theory, any $\mathfrak{T}_{e,C}$ is contained in a Siegel set of $X$, and any Siegel set of $X$ is contained in $\mathfrak{T}_{e,C}$ for some choice of $e$ and $C>0$ (see for example \cite[Prop. 2 p.18]{K90}).  Moreover, if $b$ is $(e',C')$-reduced and $e$ is a basis for which condition (3) of Definition \ref{def reduced} holds for some $C>0$, then $b$ will also be $(e,C'')$-reduced for some $C''=C''(e,C,e',C')>0$.
We are thus reduced to
proving the following:

\vskip1em

\noindent\textbf{Claim:} There is a basis $e=\{e_i\}$ of $H_\QQ$ and $C>0$ such
that $h_z$ is $(e,C)$-reduced for all $z\in\Sigma_n$.
\vskip1em

Choosing a basis $e$ for which each $e_i\in J^{\sigma_i}$ for some
$\sigma_i$, we have condition (3) in \Cref{def reduced}
by \Cref{asymptotics} since each weight filtration is centered around
0, while we may assume (2) as there are only finitely many orderings
of the basis.  By a result of Schmid \cite[Corollary 5.29]{Schmid} we
know \Cref{finitely many Siegel} is true in the $n=1$ case.  Since any Siegel set of $D$ is contained in finitely many Siegel
sets of $X$ by \Cref{orr}, it follows that $h_z$ is $(e,C_\tau)$-reduced for $p(z)$ restricted to any curve $\tau$
in $\Sigma_n$.  Taking \Cref{roughly two} into account, \Cref{check on curves} implies condition (1) for some fixed $C>0$ on all of $\Sigma_n$ and this completes
the proof.  
\end{proof}

\subsection{\Cref{definability} implies Borel's algebraicity theorem} \label{algebraic}

\begin{theor}\cite[Theor. 3.10]{Bor72}
Let $S$ be a complex algebraic variety and $f: S \lo S_{\Ga, G, K}$ a
complex analytic map to an arithmetic variety $S_{\Ga, G, K}$. Then
$f$ is algebraic.
\end{theor}

\begin{proof}
The map $f$ is a period map, hence is $\RR_{\an, \exp}$-definable by
\Cref{definability period map}. The graph of $f$ is thus a complex
analytic, $\RR_{\an, \exp}$-definable, subset of the smooth complex
algebraic manifold $S \times S_{\Ga, G, K}$. Recall the
following o-minimal Chow theorem of Peterzil-Starchenko \cite[Theor. 4.4
and Corollary 4.5]{PS} (see also \cite[Theor. 2.2]{MPT} and \cite[Theor. 2.11]{Scanlon} for more precise versions), generalizing a result of
Fortuna-\L ojasiewicz \cite{FL} in the semi-algebraic case:

\begin{theor}(Peterzil-Starchenko) \label{PS}
Let $S$ be a smooth complex algebraic manifold (hence endowing the
$\CC$-analytic manifold $S$ with a canonical $\RR_\alg$-definable
manifold structure). Let $W \subset S$ be a complex analytic
subset which is also an $\cS$-definable subset for some o-minimal
structure $\cS$ expanding $\RR_{\an}$. Then $W$ is an algebraic subset
of $S$.
\end{theor}

It follows that the graph of $f$ is an algebraic subvariety of $S \times S_{\Ga, G, K}$, hence
that $f$ is algebraic (see \cite[Prop. 8]{Serre}).
\end{proof}
%%%%%%%%%%%%%%%%%%%%%%%%%%%%%%%%%%%%%%%%%%%%%%%%%%%%%%%

\section{Algebraicity of Hodge loci: proof of \Cref{algebraicity}}

\begin{proof}[\unskip\nopunct]
We refer to \cite[Section 3.1]{K17} for the notions of (connected)
Hodge datum and morphism of (connected) Hodge data, connected Hodge varieties and Hodge
morphisms of connected Hodge varieties. Notice that any connected
Hodge variety is in particular an arithmetic quotient and that any Hodge morphism of
connected Hodge varieties is in particular a morphism of arithmetic
quotients. 

A special subvariety $Y$ of the connected Hodge variety $S_{\Ga, G, M}$ is by
definition the image $Y:= f(S_{\Ga', G', M'})$ of some Hodge morphism $f: S_{\Ga', G', M'} \lo
S_{\Ga, G, M}$. It follows from \Cref{definability}(3) and the remark
above that any special subvariety of $S_{\Ga, G, M}$ is an $\RR_{\an, \exp}$-definable subset of
$S_{\Ga, G, M}$ (endowed with its $\RR_\an$-structure of
\Cref{definability}(1). The Hodge locus $\HL(S_{\Ga, G, M})$ is defined as
the (countable) union of special subvarieties of $S_{\Ga, G, M}$.

The Hodge locus $\HL(S, \VV)$ co\"incides with the preimage $\Phi_S^{-1}
(\HL(S_{\Ga, G, M}))$. Hence to prove \Cref{algebraicity} we are
reduced to proving that the preimage $W:= \Phi^{-1}(Y)$ of any special subvariety $Y
\subset S_{\Ga, G, M}$ is an algebraic subvariety of $S$. By \Cref{definability} the period map $\Phi_S: S \lo  S_{\Ga, G, M}$
is $\RR_{\an, \exp}$-definable. As $Y \subset S_{\Ga, G, M}$ is an $\RR_{\an, \exp}$-definable subset of
$S_{\Ga, G, M}$ it follows that $W=\Phi_S^{-1}(Y)$ is
an $\RR_{\an, \exp}$-definable subset of $S$ (in particular has
finitely many connected components). As $W$ is also a complex analytic
subvariety, the o-minimal Chow \Cref{PS} of Peterzil-Starchenko
implies that $W$ is an algebraic subvariety of $S$,
which finishes the proof of \Cref{algebraicity}.

\end{proof}

\begin{appendix}

\section{Real analytic manifolds with corners and definability}

\subsection{Real analytic manifolds with corners}
From the analytic point of view, the class of real analytic manifolds
with corners is natural: a compact real-analytic manifold with corners
is the real version of the compactification of a complex analytic manifold
by a normal crossing divisor. However this class of
manifolds has been poorly studied and even their definition is not 
universally agreed. We use the one given by \cite{Dou}, which has been
clarified and developed in \cite{Joyce}. For the convenience of the reader we recall the basic
definitions but we refer to \cite{Joyce} for more details. Notice
that Joyce works in the $C^\infty$ context, but all the definitions we
need translate literally to the real-analytic setting by replacing ``smooth'' with ``real-analytic''.

Let $X$ be a paracompact Hausdorff topological space $X$ and $n\geq 1$
an integer. An $n$-dimensional chart with corners on $X$ is a pair 
$(U, \varphi)$ where $U$ is an open subset in $\RR^n_k:= \RR_{\geq 0}^k
\times \RR^{n-k}$ for some $0 \leq k \leq n$ and $\varphi: U \lo X$ is a homeomorphism with a
non-empty open set $\varphi(U)$.

Given $A \subset \RR^m$ and $B \subset \RR^n$ and $\alpha:A \lo B$
continuous, we say that $\alpha$ is real-analytic if it extends to a
real-analytic map between open neighborhoods of $A$, $B$.

Two $n$-dimensional charts with
corners $(U, \varphi)$, $(V, \psi)$ on $X$ are said real-analytically compatible if $\psi^{-1}
\circ \varphi: \varphi^{-1}(\varphi(U) \cap \psi(V)) \lo
\psi^{-1}(\varphi(U) \cap \psi(V))$ is a homeomorphism and $\psi^{-1}
\circ \varphi$ (resp. its inverse) are real-analytic in the sense
above.

An $n$-dimensional real analytic atlas with corners for $X$ is a system $\{(U_i,
\varphi_i): i \in I\}$ of pairwise real-analytically compatible charts
with corners on $X$ with $X= \bigcup_{i\in I} \varphi_i(U_i)$. We call
such an atlas maximal if is not a proper subset of any other
atlas. Any atlas is contained in a unique maximal atlas: the set of
all charts with corners $(U, \varphi)$ on $X$ compatible with $(U_i,
\varphi_i)$ for all $i \in I$. 

A real-analytic manifold with corners of dimension $n$ is a paracompact Hausdorff
topological $X$ equipped with a maximal $n$-dimensional real-analytic
atlas with corners. Weakly real-analytic maps between real-analytic
manifolds with corners are the continuous maps which are real-analytic
in charts (cf. \cite[def. 3.1]{Joyce}, where a stronger notion of
real-analytic map is also defined; we won't need this strengthened notion).

Given $X$ a real-analytic $n$-manifold with corners, one defines its
boundary $\partial X$ (cf. \cite[def. 2.6]{Joyce}. This is a
real-analytic $n$-manifold with corners for $n>0$, endowed with an
immersion (not necessarily injective) $i_X: \partial X
\lo X$ (cf. \cite[prop.2.7]{Joyce}) which is real-analytic
(\cite[Theor. 3.4.(iv)]{Joyce}) in particular weakly real-analytic.

\subsection{$\cR$-definable manifolds with corners} \label{definable corner}
Let $\cR$ be any fixed o-minimal expansion of $\RR$. The notion of
$\cR$-definable manifold is given in \cite[chap.10]{VDD} and in
\cite[p.507]{VdM2}. We will need the extended notion of
$\cR$-definable manifold with corners.

Let $X$ be a paracompact Hausdorff
topological space $X$. An $n$-dimensional chart with
corners $(U, \varphi)$ on $X$ is said to be $\cR$-definable if $U$ is an $\cR$-definable subset of
$\RR^n$ (equivalently: of $\RR^n_k$).

Two $n$-dimensional $\cR$-definable charts with
corners $(U, \varphi)$, $(V, \psi)$ on $X$ are said $\cR$-compatible if $\psi^{-1}
\circ \varphi: \varphi^{-1}(\varphi(U) \cap \psi(V)) \lo
\psi^{-1}(\varphi(U) \cap \psi(V))$ is an $\cR$-definable
homeomorphism between $\cR$-definable subsets 
$\varphi^{-1}(\varphi(U) \cap \psi(V))$ and $\psi^{-1}(\varphi(U) \cap
\psi(V))$ of $\RR^n$. 

An $n$-dimensional $\cR$-definable atlas with corners for $X$ is a system $\{(U_i,
\varphi_i): i \in I\}$, {\em $I$ finite}, of pairwise $\cR$-compatible $\cR$-definable charts
with corners on $X$ with $X= \cup_{i\in I} \varphi_i(U_i)$. Two such
atlases $\{(U_i, \varphi_i): i \in I\}$ and$\{(V_j,
\psi_j): j \in J\}$  are said $\cR$-equivalent if all the ``mixed''
transition maps $\psi_j \circ \varphi_i^{-1}$ are $\cR$-definable.

An $\cR$-definable manifold with corners of dimension $n$ is a paracompact Hausdorff
topological $X$ equipped with an $\cR$-equivalence class of $n$-dimensional $\cR$-definable
atlas with corners.

\begin{rem}
Notice that the definitions of real-analytic manifold with corners
and $\cR$-definable manifold with corners are parallel,
except the crucial fact that we work in a strictly finite setting for
$\cR$-definable manifolds: the set $I$ of charts has to be finite. This finiteness
condition, in addition to the definability condition, ensures the
tameness at infinity of the $\cR$-definable manifolds with corners.
\end{rem}

We say that a subset $Z\subset X$ is $\cR$-definable (resp. open or closed)
if $\varphi_i^{-1}(Z \cap \varphi_i(U_i))$ is an $\cR$-definable (resp. open or closed) subset
of $U_i$ for all $i \in I$. An $\cR$-definable map between 
$\cR$-definable manifolds (with corners) is a map whose graph is an $\cR$-definable subset of the
$\cR$-definable product manifold (with corners).

\subsection{Compact real-analytic manifolds with corners are
$\RR_{\an}$-definable}

\begin{prop} \label{ManifoldWC}
Let $X$ be a compact real-analytic $n$-manifold with corners. Then $X$ has
a natural structure of $\RR_{\an}$-definable manifold with
corners. Moreover the map $i_X: \partial X \lo X$ is
$\RR_{an}$-definable. In particular the interior $X \setminus
i_X(\partial X)$ is an $\RR_{\an}$-definable manifold.
\end{prop}

\begin{proof}
For each point $x$ of $X$ choose $\varphi_x: U_x \lo (X, x)$ a real-analytic
chart with corners whose image $\varphi(U_x)$ is a neighborhood of
$x$. Without loss of generality we can assume that $U_x \subset \RR^n_k$ is relatively
compact and semi-analytic, hence $\RR_{\an}$-definable. Hence $(U_x,
\varphi_x)$ is a real-analytic chart with corners for $X$ which is
also an $\RR_\an$-definable chart with corners for $X$. 

Fix $x, y$ two points in $X$. The fact that the two real-analytic charts $(U_x, \varphi_x)$ and $(U_y,
\varphi_y)$ are real-analytically compatible implies immediately that
they are $\RR_{\an}$-compa\-tible.

The space $X$ is compact hence one can extract from the covering
family $\{(U_x, \varphi_x),\; x \in X\}$ a finite subfamily $\{ (U_i,
\varphi_i), \; i \in I\}$, such that $X = \bigcup_{i\in I}
\varphi_i(U_i)$: this is an $n$-dimensional
$\RR_{\an}$-definable atlas with corners for $X$, which defines a
structure of $\RR_\an$-definable manifold with corners on $X$.

One easily checks that this structure is independent of the choice of
the finite extraction $\{ (U_i,
\varphi_i), \;i \in I\}$ of $\{(U_x, \varphi_x), \;x \in X\}$, and
also of the choice of the relatively compact and semi-analytic subsets
$U_x$.

Hence $X$ has a natural structure of $\RR_{\an}$-definable manifold with
corners. The same procedure endows the compact real-analytic $(n-1)$-manifold
with corners $\partial X$ with a natural $\RR_\an$-definable
structure. The fact that $i_X: \partial X \lo X$ is weakly
real-analytic implies immediately that $i_X$ is $\RR_\an$-definable
and that the manifold $X \setminus i_X(\partial X)$ is $\RR_\an$-definable.
\end{proof}
\end{appendix}

\sspace
\noindent Benjamin Bakker : University of Georgia

\noindent email : \texttt{bakker.uga@gmail.com}.

\sspace
\noindent Bruno Klingler : Humboldt Universit\"at zu Berlin

\noindent email : \texttt{bruno.klingler@hu-berlin.de}.

\sspace
\noindent Jacob Tsimerman : University of Toronto

\noindent email : \texttt{jacobt@math.toronto.edu}.

\end{document}